\documentclass[12pt]{amsart}
\usepackage{amsmath}
\usepackage{amsmath,amssymb,enumitem,ulem,multicol}
\usepackage{mathabx}
\usepackage{caption}
\usepackage{subcaption}
\usepackage{amssymb}
\usepackage{amsthm}
\usepackage{parskip}
\usepackage[autostyle]{csquotes}
\usepackage{mathrsfs}
\usepackage{mathtools}
\usepackage{csquotes}
\usepackage[utf8]{inputenc}
\usepackage[english]{babel}
\usepackage{graphicx}
\usepackage{xcolor}
\usepackage{epstopdf}
\usepackage{esint}

\newcounter{example}[section]
\newenvironment{example}[1][]{\refstepcounter{example}\par\medskip
	\textbf{Example~\thesection.\theexample. #1} \rmfamily}{\medskip}
\newtheorem{theorem}{Theorem}[section]

\newtheorem{lemma}[theorem]{Lemma}

\theoremstyle{definition}

\newtheorem{remark}{Remark}[section]

\newtheorem{exmp*}{Example}
\theoremstyle{remark}

\begin{document}

\allowdisplaybreaks[4]
\numberwithin{figure}{section}
\numberwithin{table}{section}
 \numberwithin{equation}{section}

\title[HHO method for elliptic obstacle problem]
{A posteriori error analysis of hybrid higher order methods for the elliptic obstacle problem}
 \author{Kamana Porwal}
 \email{kamana@maths.iitd.ac.in}
 \address{Department of Mathematics, Indian Institute of Technology Delhi, New Delhi - 110016}
\author{Ritesh Singla}\thanks{The second author's work is supported by institute fellowship by IIT Delhi, India}
\email{iritesh281@gmail.com}
 \address{Department of Mathematics, Indian Institute of Technology Delhi, New Delhi - 110016}

 \date{}
 \begin{abstract}
{In this article, $a~posteriori$ error analysis of the} elliptic obstacle problem  is addressed using hybrid high-order methods. The method involve cell unknowns represented by degree-$r$ polynomials and face unknowns represented by degree-$s$ polynomials, where $r=0$ and $s$ is either $0$ or $1$.  The discrete obstacle constraints are specifically applied to the cell unknowns. The analysis hinges on the construction of a suitable Lagrange multiplier, a residual functional and a linear averaging map. {The reliability and the efficiency of the proposed $a~posteriori$ error estimator is discussed,  and the study is concluded} by numerical experiments supporting the theoretical results.
 \end{abstract}
 
 
\maketitle

\def\R{\mathbb{R}}
\def\N{\mathbb{N}}
\def\P{\mathbb{P}}
\def\cN{\mathcal{N}}
\def\cR{\mathcal{R}}
\def\cA{\mathcal{A}}
\def\cK{\mathcal{K}}
\def\cT{\mathcal{T}}
\def\fT{\mathfrak{T}}
\def\fR{\mathfrak{R}}
\def\cE{\mathcal{E}}
\def\cF{\mathcal{F}}
\def\cH{\mathcal{H}}
\def\cG{\mathcal{G}}
\def\cK{\mathcal{K}}
\def\cR{\mathcal{R}}
\def\p{\partial}
\def\d{\mathrm{d}}
\def\sjump#1{[\hskip -1.5pt[#1]\hskip -1.5pt]}
\def\saver#1{\{\hskip -1.5pt\{#1\}\hskip -1.5pt\}}
\def\[{\partial}
\def\O{\Omega}

\section{Introduction}\label{sec1}
\par
\noindent
The elliptic obstacle problem is a fundamental mathematical model that arises in various fields of applied mathematics, including engineering, physics, and economics, having a wide ranging applications. It is used in shape optimization, where one seeks the optimal shape of a structure under certain constraints. In mathematical terms, the problem seeks a function that satisfies a partial differential equation, remains equal to or above the obstacle function, and minimizes a certain energy. In the 1960s, Guido Stampacchia made groundbreaking contributions to the theory of elliptic obstacle problems \cite{stamp}. His research  provided a detail understanding of solution's  existence, its behavior, and the characterization of corresponding free boundary regions. To efficiently approximate solutions to these problems, various numerical methods have been developed, including finite element methods, finite difference methods, and variational approaches. The articles \cite{32,33,16,9} offer $a~priori$ error analysis of finite element methods for the obstacle problem. Additionally, $a~posteriori$ error analysis for the elliptic obstacle problem using various finite element methods are developed in \cite{34, 35, 5, 36, 37, 38, 39, 13}. In \cite{13}, Veeser proposed a residual based approach for $a~posteriori$ error analysis for elliptic obstacle problem which was later adopted by authors in \cite{37} to derive $a~posteriori$ error analysis using higher order finite element method in energy norm. 
\par
In recent progressions, the field has witnessed the emergence of higher-order finite element methods. These methods offer solutions that are not only more precise but also come with reduced computational cost. The hybrid higher order (HHO) method, originally developed in the context of linear elasticity by Pietro and Ern in \cite{2}, has been extended to formulate an arbitrary-order primal method for the model poisson problem \cite{1}. Subsequently, in \cite{3}, additional challenges were introduced by imposing constraints that the solution must satisfy, marking the transition from the Poisson problem to the obstacle problem. These works primarily focus on $a~ priori$ error analysis in the energy norm using HHO methods. HHO methods have been further extended to address various linear partial differential equations (PDEs), including advection-diffusion \cite{18}, elliptic interface problems \cite{20}, and Stokes equations \cite{19}. Additionally, extensions to nonlinear PDEs, such as steady incompressible Navier-Stokes equations, are discussed in \cite{22}. Lowest-order HHO methods are closely related to the hybrid finite volume method \cite{23} and the mimetic finite difference methods \cite{24,25}. The discrete space employed in our article draws inspiration from the work of {Cicuttin \cite{3}}. The method involve cell unknowns represented by degree-$r$ polynomials and face unknowns represented by degree-$s$ polynomials, where $r=0$ and $s$ is either $0$ or $1$. The imposition of discrete obstacle constraints specifically targets the cell unknowns. {The core of the analysis relies upon the construction of a suitable discrete Lagrange multiplier, a residual functional and a linear averaging operator which maps the functions from the discontinuous finite element space to the conforming finite element space.} The main results of this article are presented in Theorem 3.1, where we establish the reliability of the proposed error estimator in the energy norm and Theorem 3.2, where we demonstrate the efficiency of the error estimator.
\par
Let $\Omega$ represent a bounded polygonal/polyhedral domain in $\mathbb R^d$, with $d\in\{2,3\}$. When considering a subset $X$ of $\Omega$, we represent the standard inner product and norm in $L^2(X)$ by $(\cdot,\cdot)_X$ and $\lVert\cdot\rVert_X$ respectively, with the understanding that the subscript is omitted when $X=\Omega$ and that the same notation is used in $L^2(X)^d$. In the subsequent analysis, we adhere to the standard notations of the Sobolev space, where $W^{m,p}(\Omega)$ signify the Sobolev spaces equipped with the norm $\lVert\cdot\rVert_{m,p}$ and seminorm $\lvert\cdot\rvert_{m,p}$. When $p=2$, we represent the space $W^{m,p}(\O)$ by $H^m(\O).$
\par
Let $f\in L^2(\O)$ denotes the forcing term and we consider the obstacle $\chi\in H^1(\Omega)\cap C^{0}(\overline \Omega)$, with the constraint that $\chi\leq 0$ on $\partial \Omega$. Define the set
\begin{equation}\label{cK}
\cK\coloneqq\{v\in H^1_0(\O)\colon v\geq \chi .~ \text{a.e in}~\O \}. 
\end{equation}
Since $\chi^+=\max\{\chi,0\}\in \cK$, it is worth noting that $\cK$ is a non-empty, closed, and convex subset of $H_0^1(\Omega)$.
\par
The continuous weak formulation of the obstacle problem is to find $u\in \cK$, satisfying
\begin{equation} \label{2}
a(u,v-u)\geq (f,v-u) ~\forall v\in \cK,
\end{equation}
where $a(v,w)\coloneqq (\nabla v,\nabla w)~\text{for all}~v,w\in H^1_0(\O)$. Since $a(\cdot,\cdot)$ defines a continuous, coercive bilinear form on $H^1_0(\O)\times H^1_0(\O)$, the existence of a unique solution to \eqref{2} is established by the result of Stampacchia \cite{9, 17}.
\par
We next introduce a functional $\sigma(u)\in H^{-1}(\Omega)$, which characterizes the residual of $u$ concerning the variational formulation of \eqref{2} across the entire domain $\Omega$. For all $\phi\in H^1_0(\Omega)$, this functional is defined as:
\begin{equation}\label{3}
\langle\sigma(u),\phi\rangle_{-1}\coloneqq (f,\phi)-(\nabla u,\nabla\phi).
\end{equation}
In this equation, the symbol $\langle\cdot,\cdot\rangle_{-1}$ represents the duality pairing between $H^{-1}(\Omega)$ and $H^1_0(\Omega)$. The assurance of the existence of $\sigma(u)$ is established by the Riesz representation theorem \cite{kind}.
\begin{remark}\label{r1}
From \eqref{2}, and \eqref{3}, we deduce that for any $v\in\cK$,
\begin{align}\label{ssign}
\langle \sigma(u),v-u\rangle_{-1}\leq 0.
\end{align}
\end{remark}
\par
The rest of the article is organized as follows. In Section $2$, we provide an introduction to the method, covering concepts such as spaces of degrees of freedom (DOFs),  the discrete gradient reconstruction operator, the discrete problem, and an exploration of its well-posedness. The construction and properties of the discrete Lagrange multiplier are also discussed in Section $2$. Section $3$ is dedicated to $a~ posteriori$ error analysis, where we define the residual functional and subsequently establish our main results: the reliability of the error  estimator in Theorem $3.1$ and the efficiency of estimator in Theorem $3.2$. Finally in Section $4,$ we present numerical results in two dimension validating the theoretical performance of the $a~ posteriori$ error estimator.

\section{Hybrid High-Order Discretization}\label{sec2}
Within this section, we present the methodology for discretizing the elliptic obstacle problem through the hybrid high-order approach.
\subsection{Discrete Setting} ~
\par
We examine a sequence of meshes that become progressively finer, denoted as $(\cT_h)_{h>0}$, where the parameter $h$ signifies the mesh size and decreases as the refinement progresses. For all $h>0$, we assume that the closure of the domain $\overline{\O}$ is the union of the closures of individual elements $T$ in the mesh, and we define $h$ as the maximum diameter among all elements $T\in\cT_h$, with $h_T$ representing the diameter of each element. In this context, we characterize the faces as hyperplanar closed connected subsets of $\O$ with positive $(d-1)$-dimensional measure. All the faces fall into one of two categories: $(i)$ they exist within the boundaries of two distinct elements, representing an interface $i.e.$ there exist $T,\widetilde{T}\in\cT_h$ such that $F\subset \partial T\cap\partial \widetilde{T}$, or $(ii)$ they lie on the boundary of a single element and the domain, representing a boundary face $i.e.~$there exists $T\in\cT_h$ such that $F\in\partial T\cap\partial \O$. The interior interfaces are collectively referred to as $\cF^i_h$, while boundary faces are grouped under $\cF^b_h$. We define $\cF_h$ as the union of these two sets. The diameter of a face $F\in \cF_h$ is denoted by $h_F$.
\par  For each element $T\in \cT_h$, we define $\cF_T$ as the set of faces constituting the boundary of $T$. Moreover, for any face $F\in \cF_h$, we use ${\bf n}_{TF}$ to denote the unit normal vector pointing outward from $T$. We also assume that the triangulation $\mathcal{T}_h$ is regular, signifying the existence of a constant $C>0$ such that
\begin{align*}
	\dfrac{h_T}{\rho_T} \leq C\quad\forall ~T \in \mathcal{T}_{h},
\end{align*} 
where $\rho_T$ represents the diameter of the largest ball inscribed in the element $T$ \cite{5}. For a shape-regular mesh sequence $(\cT_h)_{h>0}$, the maximum number of faces of a mesh cell is uniformly bounded (see \cite[Lemma 1.41]{6}). In other words, there exists a positive integer $N_\partial$, uniform with respect to $h$, such that
$$\max_{T\in\cT_h}\text{card}(\cF_T)\leq N_\partial~~~\forall h>0.$$

Here onwards, we adopt the symbol $a\lesssim b$ to signify the existence of a positive constant $C$ satisfying $a\leq Cb$. This constant's value may vary on each instance but remains independent of both the mesh cell $T\in\cT_h$ and $h$. 
\par
In view of these notations, we revisit the following trace inequalities that will be used frequently in the later analysis.
\begin{lemma}[Trace Inequality \cite{bren}]Let $T\in\cT_h$ and let $F\in\cF_T$ denote an edge/face of the simplex $T$. The following relations hold
\begin{align}
\lVert \Phi\rVert_F&\lesssim h_F^{\frac{-1}{2}}\lVert \Phi\rVert_T\quad\forall ~\Phi\in \mathbb P^l(T),\label{ieq1}\\
\lVert \Phi\rVert_{\partial T}&\lesssim (h_T^{-1}\lVert \Phi\rVert_T^2+h_T\lVert \nabla \Phi\rVert_T^2)^{\frac{1}{2}}\quad\forall ~\Phi\in H^1(T)\label{ieq2},
\end{align}
where $\mathbb P^l(T)$ represents the space of all $d$-variate polynomial function of degree atmost $l$.
\end{lemma}

Next, we introduce the broken Sobolev space $H^1(\cT_h)$ as\vspace{0.1cm} $$H^1(\cT_h)\coloneqq \{v\in L^2(\O):v|_T\in H^1(T)~\forall T\in\cT_h\}.$$ 
\par For any $F\in \cF^i_h$, let $T$ and $\widetilde{T}$ be the two simplices sharing the edge/face $F$, and ${\bf n}_{TF}$ is the unit normal vector of $F$ pointing from $T$ to $\widetilde{T}$. For $\Psi\in H^1(\cT_h)$, let $\Psi_+=\Psi|_{T}$ and $\Psi_-=\Psi|_{\widetilde{T}}$. We define the average and jump of the discontinuous function $\Psi\in H^1(\cT_h)$ on $F$ by
\begin{align*}
\saver{\Psi}\coloneqq\frac{\Psi_++\Psi_-}{2},\quad\text{and}\quad\sjump{\Psi}\coloneqq (\Psi_+-\Psi_-){\bf n}_{TF},\quad \text{respectively}.
\end{align*}
Similarly for $\boldsymbol{\Psi}\in [H^1(\cT_h)]^d$, the average and mean on $F$ is defined by
\begin{align*}
\saver{\boldsymbol{\Psi}}\coloneqq\frac{\boldsymbol{\Psi}_++\boldsymbol{\Psi}_-}{2},\quad\text{and}\quad\sjump{\boldsymbol{\Psi}}\coloneqq (\boldsymbol{\Psi}_+-\boldsymbol{\Psi}_-)\cdot{\bf n}_{TF},\quad \text{respectively}.
\end{align*}
\par  Let the edge/face $F\in\cF^b_h$ be such that  $F=\partial T\cap\partial\O$ and ${\bf n}_{TF}$ be the outside unit normal vector pointing outside $T$. Then the average and jump over the edge/face $F$ is given respectively by 
{
\begin{align*}
\saver{\Psi}&\coloneqq\Psi,\quad\text{and}\quad\sjump{\Psi}\coloneqq \Psi{\bf n}_{TF},\quad\forall ~\Psi\in H^1(\cT_h),\\
\saver{\boldsymbol{\Psi}}&\coloneqq\boldsymbol{\Psi},\quad\text{and}\quad\sjump{\boldsymbol{\Psi}}\coloneqq \boldsymbol{\Psi}\cdot{\bf n}_{TF},\quad\forall ~\boldsymbol{\Psi}\in [H^1(\cT_h)]^d.
\end{align*}
}
\subsection{Degrees of Freedom}~
\par
For  $k\in\{0,1\}$, we define the local discrete space for each $T\in\cT_h$ as follows\\
\begin{equation}\label{7}
U^k_T\coloneqq \mathbb P^0(T)\bigtimes\left\{ \prod_{F\in \mathcal F_T} \mathbb P^k(F)\right\}
\end{equation}
\par
We identify any vector $(\underline{v}_T)$ of the local discrete space $U_T^k$ as $\underline{v}_T=(v_T,(v_F)_{F\in\cF_T})$.
\par
To construct the global space of degrees of freedom, we assemble interface values from \eqref{7}\\
\begin{equation} \label{8}
U^k_h\coloneqq\left\{\prod_{T\in \cT_h}\mathbb P^0(T)\right\}\bigtimes\left\{\prod_{F\in \mathcal F_h} \mathbb P^k(F)\right\}
\end{equation}
\par
The boundary conditions are incorporated into the discrete space \eqref{8} by defining
\vspace{0.2cm}
\begin{equation}\label{9}
V_h\coloneqq \{\underline{v}_h=((v_T)_{T\in \cT_h},(v_F)_{F\in \cF_h})\in U^k_h~ |~v_F\equiv 0\text{ for all }F\in \cF_h^b\}.
\end{equation}
\par Accordingly, for each $T\in \cT_h$, we establish the restriction operator $\cR_T \colon U_h^k\rightarrow U_T^k$ mapping $U^k_h$ onto $U^k_T$, such that $\cR_T(\underline{v}_h)=\underline{v}_T$.
\par
The local interpolation operator, denoted by $I_T^k \colon H^1(T)\rightarrow U_T^k$ is defined as follows:\\
\begin{equation}\label{10}
I_T^k(v)\coloneqq (\pi_T^0v,(\pi_F^kv)_{F\in \cF_T}).
\end{equation}
\par  Here, $\pi^0_T$ represents the $L^2$-orthogonal projection onto $\mathbb P^0(T)$ and $\pi^k_F$ denotes the $L^2$-orthogonal projection onto $\mathbb P^k(F)$, where $\mathbb P^k(F)$ represents the space of all $(d-1)$ variate polynomial functions of degree at most $k$.
\par
The corresponding global interpolation operator $I_h^k \colon H^1(\O)\rightarrow U_h^k$ is expressed as, for all $v\in H^1(\O)$,\\
\vspace{-0.3cm}
\begin{equation}\label{11}
I_h^k(v)\coloneqq \big((\pi_T^0v)_{T\in \cT_h},(\pi_F^kv)_{F\in \cF_h}\big).
\end{equation}
\vspace{-0.3cm}
\begin{remark}
In case we have the non-homogeneous boundary data, say $u=g$ on $\p \O,$ then we will consider the discrete space $V_h$ as\\
\begin{equation*}
V_h\coloneqq\{\underline{v}_h\in U^k_h\colon v_F=\pi^k_F(g)~\text{for all }F\in\cF^b_h\}.
\end{equation*}
\end{remark}
\subsection{Local Gradient Reconstruction Operator}~
\par
For each $T\in\cT_h$, we introduce a local gradient reconstruction operator denoted as $p_T^{k+1}\colon U_T^k\rightarrow \mathbb P^{k+1}(T)$. This operator is defined to satisfy the following conditions. For all $\underline{v}_T\in U_T^k$,\\
\begin{align}
(\nabla p^{k+1}_T\underline{v}_T,\nabla w)_T&=(\nabla v_T,\nabla w)_T+\sum_{F\in\cF_T}(v_F-v_T,\nabla w\cdot {\bf{n}}_{TF})_F,\quad\forall ~w\in\mathbb P^{k+1}(T), \label{12}\\
\int_T{p_T^{k+1}\underline{v}_T}~dx&=\int_Tv_T ~dx.\label{13}
\end{align}
\par
Applying integration by parts in \eqref{12}, we conclude that\\
\begin{equation}
(\nabla p^{k+1}_T\underline{v}_T,\nabla w)_T=-(v_T,\Delta w)_T+\sum_{F\in\cF_T}(v_F,\nabla w\cdot {\bf{n}}_{TF})_F.\label{12n}
\end{equation}
\par
{Correspondingly, the global gradient reconstruction operator denoted as $p^{k+1}_h\colon U^k_h\rightarrow \mathbb P^{k+1}(\cT_h)$ is defined as $p^{k+1}_h\underline{v}_h|_T\coloneqq p^{k+1}_T\underline{v}_T,~\forall~\underline{v}_h\in U^k_h.$}
\vspace{0.2cm}
\begin{remark}
In the discrete space, our focus is exclusively on constant functions within the element. This leads us to the following precise definition for gradient reconstruction: \\
\begin{align}\label{gr1}
(\nabla p^{k+1}_T\underline{v}_T,\nabla w)_T=\sum_{F\in\cF_T}(v_F-v_T,\nabla w\cdot {\bf{n}}_{TF})_F.
\end{align}
\end{remark}
\vspace{-0.1cm}
\begin{remark}
For given $q\in \mathbb P^{k+1}(T)$, we have by \eqref{12n} that for all $w\in \mathbb P^{k+1}(T)\subset \mathbb P^2(T)$
\begin{align*}
(\nabla p^{k+1}_T I^k_T(q),\nabla w)_T&=-(\pi^0_Tq,\Delta w)_T+\sum_{F\in\cF_T}(\pi^k_Fq,\nabla w\cdot{\bf{n}}_{TF})_F\\
&=(q,\Delta w)_T+\sum_{F\in\cF_T}(q,\nabla w\cdot{\bf{n}}_{TF})_F\\
&=(\nabla q,\nabla w)_T.
\end{align*}
\par
Hence, we have $\nabla p^{k+1}_T(I^k_T(q))=\nabla q,~\forall~ q\in \mathbb P^{k+1}(T)$, and using \eqref{13}, we conclude that $p^{k+1}_T(I^k_T(q))=q,~\forall~ q\in\mathbb P^{k+1}(T)$.
\end{remark}
\par
Having established all the aforementioned notations, we are ready to formulate the discrete problem associated with \eqref{2}.
\subsection{Discrete Problem}~
\par
For each element $T\in\cT_h$, we define the local bilinear forms $a_T(\cdot,\cdot)$ and $s_T(\cdot,\cdot)$ on $U^k_T \times U^k_T$ as follows\\
\begin{align*}
a_T(\underline{w}_T,\underline{v}_T)&\coloneqq (\nabla p^{k+1}_T(\underline{w}_T),\nabla p^{k+1}_T(\underline{v}_T))_T+s_T(\underline{w}_T,\underline{v}_T),
\end{align*}
with,
\begin{align*}
s_T(\underline{w}_T,\underline{v}_T)&\coloneqq \sum_{F\in \cF_T}\frac{1}{h_F}(\pi^k_F(w_F-p^{k+1}_T\underline{w}_T),\pi^k_F(v_F-p^{k+1}_T\underline{v}_T))_F.
\end{align*}
\par Furthermore, the global discrete bilinear form on $U^k_h\times U^k_h$ is introduced as:\\
\begin{align*}
a_h(\underline{w}_h,\underline{v}_h)\coloneqq \sum_{T\in \cT_h}a_T(\cR_T\underline{w}_h,\cR_T\underline{v}_h)=s_h(\underline{w}_h,\underline{v}_h)+\sum_{T\in \cT_h}(\nabla p^{k+1}_T(\cR_T\underline{w}_h),\nabla p^{k+1}_T(\cR_T\underline{v}_h)),
\end{align*}
\par
where,
\begin{equation*}
s_h(\underline{w}_h,\underline{v}_h)\coloneqq\sum\limits_{T\in\cT_h}s_T(\underline{w}_T,\underline{v}_T).
\end{equation*}

\par Next, we proceed to define the local and global energy seminorms on $U^k_T$ and $U^k_h$, respectively as\\
\begin{align*}
\lVert \underline{v}_T\rVert^2_{a,T}\coloneqq a_T(\underline{v}_T,\underline{v}_T)~\forall~\underline{v}_T\in U^k_T,\quad\text{and}\quad\lVert \underline{v}_h\rVert^2_{a,h}\coloneqq \sum_{T\in \cT_h}\lVert \cR_T\underline{v}_h\rVert^2_{a,T}~\forall~\underline{v}_h\in U^k_h.
\end{align*}
\begin{remark}
{
From here onwards, for any $\widetilde{\cT}_h\subset\cT_h$ and $\underline{v}_h\in U^k_h$, we define $\lVert \underline{v}_h\rVert_{a, \widetilde{\cT}_h}$ as 
\begin{equation*}
\lVert \underline{v}_h\rVert_{a,\widetilde{\cT}_h}^2\coloneqq \sum_{T\in\widetilde{\cT}_h}\lVert \cR_T\underline{v}_h\rVert^2_{a,T}.
\end{equation*}
}
\end{remark}
The closed convex subset of $V_h$ is defined in the following manner: \\
\begin{equation*}
\cK_h\coloneqq\{\underline{v}_h\in V_{h} : (v_T,1)_T \geq(\chi,1)_T~\forall ~T\in \cT_h\}.
\end{equation*} 
\par The discrete problem reads, find $\underline{u}_h\in \cK_h$ such that\\
\begin{equation}\label{14}
a_h(\underline{u}_h,\underline{v}_h-\underline{u}_h)\geq \sum_{T\in \cT_h} (f,v_T-u_T)_T \quad\forall~ \underline{v}_h\in \cK_h.
\end{equation}
\par To establish the well-posedness of the discrete problem \eqref{14}, we examine the coercivity and boundedness of the discrete bilinear form $a_h(\cdot,\cdot)$ on $U^k_h\times U^k_h$. To this purpose, we equip the space $V_{h}$ with the following norm:\\
\begin{equation}\label{dnorm}
\lVert \underline{v}_h\rVert^2_{h}\coloneqq \sum_{T\in\cT_h} \sum_{F\in\cF_T} h_F^{-1}\lVert v_F-v_T\rVert^2_F.
\end{equation}
\par From \cite[Lemma 3.4]{3}, $a_h(\cdot,\cdot)$ is bounded and coercive with respect to $\lVert\cdot\rVert_{h}$ on $V_h$, therefore from Stampacchia \cite{9}, there exists a unique solution of the discrete problem \eqref{14}.
\par
Since the discrete constraints are imposed on the elements, we further divide the set $\cT_h$ into the discrete contact set $\mathbb F_h$ and the non-contact set $\mathbb N_h$, in accordance with the behaviour of  the discrete solution $\underline{u}_h$ as:\\
\begin{equation}\label{set}
\mathbb F_h\coloneqq\{T\in \cT_h:(u_T,1)_T=(\chi,1)_T\}~~\text{and}~~\mathbb N_h\coloneqq\{T\in \cT_h : (u_T,1)_T>(\chi,1)_T\}.
\end{equation}
\par Moreover, for $\underline{v}_h$ and $\underline{w}_h$ belonging to the function space $U^k_h$, we define the inner product $\langle \cdot,\cdot\rangle_h$ on $U^k_h\times U^k_h$ as follows:
\begin{equation}\label{dips}
\langle\underline{v}_h,\underline{w}_h\rangle_h\coloneqq \sum_{T\in \cT_h}\int_T v_T w_T~d x+\sum_{F\in\cF_h}\int_F v_Fw_F~ds.
\end{equation}
\par Subsequently, we introduce a discrete Lagrange multiplier $\underline{\sigma}_h$, which serves as a discrete counterpart to the residual $\sigma(u)$. We define $\underline{\sigma}_h=\big((\sigma_T)_{T\in\cT_h},(\sigma_F)_{F\in\cF_h}\big)\in U^k_h$ as\\
\begin{equation}\label{dlm}
\langle\underline{\sigma}_h,\underline{\phi}_h\rangle_h\coloneqq\sum_{T\in \cT_h}(f,{\phi}_T)-a_h(\underline{u}_h,\underline{\phi}_h)~~\forall \underline{\phi}_h\in U^k_h.
\end{equation}
\par In the following lemma, we derive the relations satisfied by the discrete Lagrange multiplier $\underline{\sigma}_h$. For the sake of simplicity, we present a detailed proof for the case of two dimensions ($d=2$), employing linear polynomials on the faces. It is noteworthy that the analogous proof can be derived seamlessly to the three-dimensional scenario $(d=3)$.
\begin{lemma}\label{l2.2n}~
\begin{enumerate}
\item For each element $T\in\cT_h$, it holds that
\begin{align*}
{\sigma}_T  \leq 0&\quad  \forall ~T\in \cT_h.
\end{align*}
\item For each element $T\in\mathbb N_h$, we have
\begin{align*}
{\sigma}_T=0&\quad\forall~T\in \mathbb N_h.
\end{align*}
\item For each $F\in\cF_h$, it holds that
\begin{align*}
{\sigma}_F=0&\quad\forall~F\in\cF_h.
\end{align*}
\end{enumerate}
\end{lemma}
\begin{proof} Consider the sets $\{T_i\}_{i=1}^n$ encompassing all elements in $\cT_h$ and $\{F_j\}_{j=1}^m$ representing all faces in $\cF_h.$ The total degrees of freedom will be $n+2m$ and let $\textbf{e}_k,~1\leq k\leq n+m$, refers to the $k-$th standard basis vector in $\mathbb R^{n+m}$. {These degrees of freedom play a crucial role in characterizing the solution space, with each $\textbf{e}_k$ contributing to the $k-$th component of the discrete solution.}
\begin{enumerate}
\item Let $T\in\cT_h$ be fixed, and $T=T_i$ for some $1\leq i\leq n.$ We consider the test function $\underline{\phi}_h=\textbf{e}_i\in U^k_h$ and let $\underline{v}_h=(\underline{u}_h+\underline{\phi}_h)\in \cK_h.$ It then follows from equation \eqref{dips}, \eqref{dlm} and \eqref{14} that
\begin{align*}
\lvert T\rvert*\sigma_T&=\langle \underline{\sigma}_h,\underline{\phi}_h\rangle_h\\
&=\sum_{T\in \cT_h}(f,{\phi}_T)-a_h(\underline{u}_h,\underline{\phi}_h)\\
&=\sum_{T\in \cT_h}(f,{v}_T-{u}_T)-a_h(\underline{u}_h,\underline{v}_h-\underline{u}_h)\leq 0.
\end{align*}
\item Assuming a fixed $T\in\mathbb N_h$ and let $\epsilon>0$ be such that $(u_T,1)_T\pm \epsilon >(\chi,1)_T.$ Let $1\leq i\leq n$ be such that $T=T_i.$ Consider the test function $\underline{\phi}_h=\frac{1}{\lvert T\rvert}\textbf{e}_i\in U^k_h$. Then $\underline{v}^{\pm}_h=\underline{u}_h\pm \epsilon \underline{\phi}_h\in\cK_h$ and from equation \eqref{dips}, \eqref{dlm} and \eqref{14} we conclude
\begin{align*}
\pm \epsilon\sigma_T&=\langle \underline{\sigma}_h,\pm\epsilon\underline{\phi}_h\rangle_h\\
&=\langle \underline{\sigma}_h,\underline{v}^{\pm}_h-\underline{u}_h\rangle_h\\
&=\sum_{T\in \cT_h}(f,{v}^{\pm}_T-{u}_T)-a_h(\underline{u}_h,\underline{v}^{\pm}_h-\underline{u}_h)\leq 0.
\end{align*}
Therefore, $\sigma_T=0$ for all $T\in\mathbb N_h.$
\item Let $F\in\cF_h$ be fixed and $F=F_j$ for some $1\leq j\leq m$. Let $\lambda_{1j}$ and $\lambda_{2j}$ be the barycentric coordinates associated with the two vertices of edge $F_j.$ Consider the test functions $\underline{\phi}_h=\lambda_{1j}\textbf{e}_{n+j}$ and $\underline{\psi}_h=\lambda_{2j}\textbf{e}_{n+j}$ in $U^k_h$. Then $\underline{v}^{\pm}_h=\underline{u}_h\pm \underline{\phi}_h$ and $\underline{w}^{\pm}_h=\underline{u}_h\pm \underline{\psi}_h$ belong to $\cK_h$ and from equation \eqref{dips}, \eqref{dlm} and \eqref{14} we obtain
\begin{align*}
\pm(\sigma_{F_j},\lambda_{1j})_{F_j}&=\langle \underline{\sigma}_h,\pm\underline{\phi}_h\rangle_h\\
&=\langle \underline{\sigma}_h,\underline{v}^{\pm}_h-\underline{u}_h\rangle_h\\
&=\sum_{T\in \cT_h}(f,{v}^{\pm}_T-{u}_T)-a_h(\underline{u}_h,\underline{v}^{\pm}_h-\underline{u}_h)\leq 0.
\end{align*}
{Therefore, $(\sigma_{F_j},\lambda_{1j})_{F_j}=0$. Similarly we can prove that $(\sigma_{F_j},\lambda_{2j})_{F_j}=0$ by taking test functions to be $\underline{w}^{\pm}_h$. Since $\sigma_{F_j}$ is a linear polynomial over each face, we conclude that $\sigma_{F_j}=0.$}
\end{enumerate}
\end{proof}
\par Next, we introduce a functional $F_{\underline{\sigma}_h}\in H^{-1}(\O)$ as: \\
\begin{equation}\label{15}
\langle F_{\underline{\sigma}_h},v\rangle_{-1}\coloneqq\sum_{T\in \cT_h}\int_T {\sigma}_T v~dx\quad\forall ~v\in H^1_0(\O).
\end{equation}
\vspace{-0.1cm}
\begin{remark}
The motivation behind formulating $F_{\underline{\sigma}_h}$ in this manner is two folds. First, it ensures alignment with $\langle \underline{\sigma}_h,\underline{v}_h\rangle_h~\forall~ \underline{v}_h\in V_h$. Second, it will be helpful in estimating the error between the discrete Lagrange multiplier $\underline{\sigma}_h$ and the continuous Lagrange multiplier $\sigma(u)$.
\end{remark}
Define the space $V_E\coloneqq  H^1_0(\O)+\mathbb P^{k+1}(\cT_h)$. {We introduce a new mesh dependent bilinear} form $b_h(\cdot,\cdot)$ on $V_E\times V_E$ as\\
\begin{equation}\label{4.28}
b_h(w,v)=\int_{\O} \nabla_hw\cdot\nabla_hv~dx=\sum_{T\in \cT_h}\int_T \nabla w\cdot\nabla v~dx, \quad \forall ~w,v \in V_E .
\end{equation}
\vspace{-0.3cm}
\begin{remark}\label{r2.5}
It is to be noted that for $w,v\in H^1(\O)$, we have $b_h(w,v)=a(w,v)$ and for $\underline{w}_h,\underline{v}_h\in V_h$, we have $$a_h(\underline{w}_h,\underline{v}_h)=b_h(p_h^{k+1}\underline{w}_h,p_h^{k+1}\underline{v}_h)+s_h(\underline{w}_h,\underline{v}_h).$$
\end{remark}
\subsection{Averaging Operator}\label{ss2.5}~
\par
Let an integer $l\geq 1$ be fixed. The node-averaging operator $E^l_h :\mathbb P^l(\cT_h)\rightarrow \mathbb P^l(\cT_h)\cap H^1_0(\O)$ is defined by setting for each (Lagrange) interpolation node $N$ in the following way
\begin{equation}\label{enrch}
E_h^lv_h(N)\coloneqq\begin{cases}
\frac{1}{\text{card}(\cT_N)}\sum_{T\in \cT_N}(v_h)_{|T}(N), & \text{ if } N\in \O,\\
0, & \text{ if } N\in \partial\O,
\end{cases}
\end{equation}
where the set $\cT_N\subset \cT_h$ collects all the simplices to which $N$ belongs and $\mathbb P^l(\cT_h)\coloneqq\{v\in L^2(\O)\colon v|_T\in\mathbb P^l(T)~\forall~T\in\cT_h\}$. We then set
$$u^*_h\coloneqq E^{k+1}_h p^{k+1}_h\underline{u}_h\in H^1_0(\O).$$
\par The subsequent outcome addresses the valuable approximation characteristics of the averaging map.
\begin{lemma} \label{l2.1}
Let $E^{k+1}_h$ be the operator defined from \eqref{enrch}. Then for any $v\in \mathbb P^{k+1}(\cT_h)$ it holds
\begin{eqnarray}
\lVert E^{k+1}_h v-v\rVert^2_T+\lVert h_T\nabla(E^{k+1}_h v-v)\rVert^2_T \lesssim \sum_{F\in\widetilde{\cF}_{T}}h_F\lVert \sjump{v}\rVert^2_F,
\end{eqnarray}
where $\widetilde{\cF}_T\coloneqq \{F\in\cF_h:{F}\cap\partial T\neq \emptyset\}$.
\end{lemma}
These approximation properties can be proved by using the scaling arguments \cite{31}.
\section{A Posteriori Error Analysis}\label{sec3}
We introduce the residual functional $\fR_h\in H^{-1}(\O)$, which holds a crucial significance in subsequent analysis and is defined as for all $\phi\in H^1_0(\O)$,\\
 \begin{align}
\langle \fR_h,\phi\rangle_{-1}\coloneqq(\nabla_h u,\nabla_h\phi)-(\nabla_hp^{k+1}_h\underline{u}_h,\nabla_h\phi)+\langle\sigma(u)-F_{\underline{\sigma}_h},\phi\rangle_{-1}.\label{20}
\end{align}
\par A use of integration by parts yields\\
\begin{align}
\langle \fR_h,\phi\rangle_{-1}=\sum_{T\in \cT_h}\Big(\big(f+\Delta p^{k+1}_T\underline{u}_T-\sigma_T,\phi\big)_T-\sum_{F\in \cF_T} \big(\nabla p^{k+1}_T\underline{u}_T \cdot {\bf n}_{TF},\phi\big)_F\Big)\label{21}.
\end{align}
\subsection{Reliability of the Error Estimator}\label{ss3.1}~
\par
Define the following contributions of the error estimator:
\begin{align*}
\eta_1&\coloneqq \left(\sum_{F\in \cT_h}\left\Vert\nabla(p^{k+1}_T\underline{u}_T-u^*_h)\right\Vert^2_T\right)^{\frac{1}{2}},\\
\eta_2&\coloneqq \left(\sum_{T\in \cT_h}h_T^2\left\Vert\left(f+\Delta p^{k+1}_T\underline{u}_T-\sigma_T\right)-\pi^0_T\left(f+\Delta p^{k+1}_T\underline{u}_T-\sigma_T\right)\right\Vert^2_T\right)^{\frac{1}{2}},\\
\eta_3&\coloneqq \left(\sum_{T\in \cT_h}s_T(\underline{u}_T,\underline{u}_T)\right)^{\frac{1}{2}}.
\end{align*}
\par
Using residual functional $\fR_h$, we derive below an estimate for the error.
\begin{lemma}\label{l1}
It holds that 
\begin{eqnarray*}
\lVert \nabla_h(u-p_h^{k+1}\underline{u}_h)\rVert^2+\lVert \sigma(u)-F_{\underline{\sigma}_h}\rVert^2_{-1}\lesssim (\lVert \fR_h\rVert^2_{-1}+\lVert \nabla_h(p^{k+1}_h\underline{u}_h-u^*_h)\rVert^2)-\\12\langle\sigma(u)-F_{\underline{\sigma}_h},u-u^*_h\rangle_{-1}.
 \end{eqnarray*}
 \end{lemma}
 \begin{proof}\text{Let} $u^*_h=E^{k+1}_hp^{k+1}_h\underline{u}_h$. Then using triangle's inequality, we obtain\\
\begin{align}\label{1001}
 \lVert \nabla_h(u-p_h^{k+1}\underline{u}_h)\rVert^2&\leq  2\lVert \nabla_h(u^*_h-p_h^{k+1}\underline{u}_h)\rVert^2+2\lVert\nabla_h(u-u^*_h)\rVert^2.
 \end{align}
\par Moreover, utilizing \eqref{20}, H\"older's inequality, and Young's inequality, we have\\
 \begin{align*}
 \lVert\nabla_h(u-u^*_h)\rVert^2&=(\nabla_h(u-u^*_h),\nabla_h(u-u^*_h))\\
 &=(\nabla_h(u-p^{k+1}_h\underline{u}_h),\nabla_h(u-u^*_h))+(\nabla_h(p^{k+1}_h\underline{u}_h-u^*_h),\nabla_h(u-u^*_h))\\
 &=\langle \fR_h,u-u^*_h\rangle_{-1}-\langle\sigma(u)-F_{\underline{\sigma}_h},u-u^*_h\rangle_{-1}+(\nabla_h(p^{k+1}_h\underline{u}_h-u^*_h),\nabla_h(u-u^*_h))\\
 &\leq C(\lVert \fR_h\rVert^2_{-1}+\lVert \nabla_h(p^{k+1}_h\underline{u}_h-u^*_h)\rVert^2)-2\langle\sigma(u)-F_{\underline{\sigma}_h},u-u^*_h\rangle_{-1}.
 \end{align*}
\par
Furthermore, in view of \eqref{20} and Young's inequality we have\\
 \begin{align}\label{1002}
 \lVert \sigma(u)-F_{\underline{\sigma}_h}\rVert_{-1}^2\leq 2\lVert \fR_h\rVert_{-1}^2+2\lVert \nabla_h(u-p^{k+1}_h\underline{u}_h)\rVert^2.
 \end{align}
\par Combining \eqref{1001} and \eqref{1002}, we conclude\\
\begin{eqnarray*}
 \lVert \nabla_h(u-p_h^{k+1}\underline{u}_h)\rVert^2+\lVert \sigma(u)-F_{\underline{\sigma}_h}\rVert^2_{-1}\leq C(\lVert \fR_h\rVert^2_{-1}+\lVert \nabla_h(p^{k+1}_h\underline{u}_h-u^*_h)\rVert^2)-\\12\langle\sigma(u)-F_{\underline{\sigma}_h},u-u^*_h\rangle_{-1}.
 \end{eqnarray*}
\end{proof}
We now establish a bound on the stability term $s_T(\cdot,\cdot)$ which is required in further analysis. The proof follows that of \cite[Lemma 3.1]{3}.
\begin{lemma}\label{st}
For any $\underline{v}_T\in U^k_T$, the following estimate holds $$s_T(\underline{v}_T,\underline{v}_T)\lesssim \sum_{F\in\cF_T}h^{-1}_F\lVert v_F-v_T\rVert^2_F.$$
\end{lemma}
\begin{proof}
Let $\underline{v}_T=(v_T,(v_F)_{F\in\cF_T})\in U^k_T$ be any. Then we have,\\
\begin{align*}
s_T(\underline{v}_T,\underline{v}_T)=\sum_{F\in\cF_T}\frac{1}{h_F}\lVert \pi^k_F(v_F-p^{k+1}_T\underline{v}_T)\rVert^2_F.
\end{align*}
\par By applying the triangle inequality and utilizing \eqref{13}, we obtain:\\
\begin{equation}
\label{st1}\lVert \pi^k_F(v_F-p^{k+1}_T\underline{v}_T)\rVert_F\leq \lVert v_F-v_T\rVert_F+\lVert \pi^k_F(p^{k+1}_T\underline{v}_T-\pi^0_Tp^{k+1}_T\underline{v}_T)\rVert_F.
\end{equation}
\par Further, by employing \eqref{ieq1} and \eqref{gr1}, we have:\\
\begin{align}\label{st2}
\lVert \pi^k_F(p^{k+1}_T\underline{v}_T-\pi^0_Tp^{k+1}_T\underline{v}_T)\rVert_F&\leq \lVert p^{k+1}_T\underline{v}_T-\pi^0_Tp^{k+1}_T\underline{v}_T\rVert_F\lesssim h^{\frac{-1}{2}}_F\lVert p^{k+1}_T\underline{v}_T-\pi^0_Tp^{k+1}_T\underline{v}_T\rVert_T\notag\\
&\lesssim h^{\frac{1}{2}}_F \lVert \nabla(p^{k+1}_T\underline{v}_T)\rVert_T\leq \sum_{F\in\cF_T}\lVert v_F-v_T\rVert_F.
\end{align}
\par A use of H\"older's inequality together with \eqref{st1}, and \eqref{st2}, yields the desired estimate of the lemma.
\end{proof}
In the following lemma, we derive an estimate for the residual functional $\fR_h$.
\begin{lemma}\label{l2}
It holds that
\begin{equation*}
\lVert \fR_h\rVert_{-1}\lesssim(\eta_2^2+\eta_3^2)^{\frac{1}{2}}.
\end{equation*}
\end{lemma}
\begin{proof}
Let $\phi \in H^1_0(\O)$ and $\underline{\phi}_h\in V_h$ be such that $\phi_T=\pi^0_T \phi~\forall ~T\in \cT_h$ and $\phi_F=\pi^k_F\phi~\forall ~F\in \cF_h$. 
\par Using the property of the projection operator $\pi^0_T$, equation \eqref{dips}, \eqref{dlm} and integration by parts, we have\\
\begin{align}\label{23}
\sum_{T\in \cT_h}(\pi^0_T(f+\Delta p^{k+1}_T\underline{u}_T-\sigma_T),\phi)_T&=\sum_{T\in \cT_h}(f+\Delta p^{k+1}_T\underline{u}_T-\sigma_T,\phi_T)\notag\\
&=\sum_{T\in \cT_h}(f+\Delta p^{k+1}_T\underline{u}_T,\phi_T)-\langle \underline{\sigma}_h,\underline{\phi}_h\rangle_h\notag\\
&=\sum_{T\in \cT_h}\Big[a_T(\underline{u}_T,\underline{\phi}_T)+\sum_{F\in \cF_T}(\nabla p^{k+1}_T\underline{u}_T\cdot {\bf n}_{TF},\phi_T)_F\Big].
\end{align}
\par Now using definition of $a_T(\underline{u}_T,\underline{\phi}_T)$, equation \eqref{12n} with $w=p^{k+1}_T\underline{u}_T$ and noting that $\phi_T$ is a constant on each $T$, we conclude:\\
\begin{align*}
a_T(\underline{u}_T,\underline{\phi}_T)&=(\nabla p^{k+1}_T\underline{u}_T,\nabla p^{k+1}_T\underline{\phi}_T)_T+s_T(\underline{u}_T,\underline{\phi}_T)\\
&=-(\Delta p^{k+1}_T\underline{u}_T,\phi_T)_T+s_T(\underline{u}_T,\underline{\phi}_T)+\sum_{F\in \cF_T}(\nabla p^{k+1}_T \underline{u}_T\cdot {\bf n}_{TF},\phi_F)_F\\
&=-\sum_{F\in \cF_T}(\nabla p^{k+1}_T\underline{u}_T\cdot {\bf n}_{TF},\phi_T)_F+s_T(\underline{u}_T,\underline{\phi}_T)+\sum_{F\in \cF_T}(\nabla p^{k+1}_T \underline{u}_T\cdot {\bf n}_{TF},\phi_F)_F.
\end{align*}
\par Substituting this expression back into \eqref{23} we obtain:\\
\begin{eqnarray*}\label{24}
\begin{split}
\sum_{T\in \cT_h}(\pi^0_T(f+\Delta p^{k+1}_T\underline{u}_T-\sigma_T),\phi)_T&=\sum_{T\in \cT_h}(f+\Delta p^{k+1}_T\underline{u}_T-\sigma_T,\phi_T)\\
&=\sum_{T\in \cT_h}\Big(s_T(\underline{u}_T,\underline{\phi}_T)+\sum_{F\in \cF_T}(\nabla p^{k+1}_T\underline{u}_T\cdot {\bf n}_{TF},\phi)_F\Big).
\end{split}
\end{eqnarray*}
\par Adding and subtracting $\sum\limits_{T\in\cT_h}\big(\pi^0_T(f+\Delta p^{k+1}_T\underline{u}_T-\sigma_T),\phi\big)_T$ in equation \eqref{21}, we conclude:\\
\begin{align*}\label{25}
\langle \fR_h,\phi\rangle_{-1}=&\sum_{T\in \cT_h}\bigg(f+\Delta p^{k+1}_T\underline{u}_T-\sigma_T-\pi^0_T(f+\Delta p^{k+1}_T\underline{u}_T-\sigma_T),\phi-\phi_T\bigg)_T\\&\quad\quad\quad+\sum\limits_{T\in\cT_h}\big(\pi^0_T(f+\Delta p^{k+1}_T\underline{u}_T-\sigma_T),\phi\big)_T-\sum_{T\in\cT_h}\sum_{F\in \cF_T} \big(\nabla p^{k+1}_T\underline{u}_T \cdot {\bf n}_{TF},\phi\big)_F\\
=&\sum_{T\in \cT_h}\bigg(f+\Delta p^{k+1}_T\underline{u}_T-\sigma_T-\pi^0_T(f+\Delta p^{k+1}_T\underline{u}_T-\sigma_T),\phi-\phi_T\bigg)_T\\&\quad\quad\quad+\sum_{T\in\cT_h}s_T(\underline{u}_T,\underline{\phi}_T).
\end{align*}
\par
Now, using Cauchy-Schwarz inequality we have $s_T(\underline{u}_T,\underline{\phi}_T)\leq s_T(\underline{u}_T,\underline{u}_T)^{\frac{1}{2}}s_T(\underline{\phi}_T,\underline{\phi}_T)^{\frac{1}{2}}$. From Lemma \ref{st} and equation \eqref{ieq2}, we have\\
\begin{align*}
s_T(\underline{\phi}_T,\underline{\phi}_T)^{\frac{1}{2}}&\lesssim \sum_{F\in\cF_T}h^{-1/2}_F\lVert \pi^k_F\phi-\pi^0_T\phi\rVert_F= \sum_{F\in\cF_T}h^{-1/2}_F\lVert \pi^k_F(\phi-\pi^0_T\phi)\rVert_F\\
&\lesssim  \sum_{F\in\cF_T}h^{-1/2}_F\lVert \phi-\pi^0_T\phi\rVert_F\lesssim \sum_{F\in\cF_T}h^{-1/2}_F(h_T^{-1}\lVert \phi-\pi^0_T\phi\rVert^2_T+h_T\lVert \nabla\phi\rVert^2_T)^{1/2}\\
&\lesssim \lVert\phi\rVert_{H^1(T)}.
\end{align*}
\par Therefore, $s_T(\underline{u}_T,\underline{\phi}_T)\lesssim s_T(\underline{u}_T,\underline{u}_T)^{1/2}  \lVert \phi\rVert_{H^1(T)}$.
\par
Using the above estimates and the approximation properties of the polynomial interpolation we conclude:\\
\begin{eqnarray*}
\langle \fR_h,\phi\rangle_{-1} \lesssim\sum_{T\in \cT_h}h_T\lVert f+\Delta p^{k+1}_T\underline{u}_T-\sigma_T-\pi^0_T(f+\Delta p^{k+1}_T \underline{u}_T-\sigma_T) \rVert_T\lVert \phi\rVert_{H^1(T)}\\+\sum_{T\in \cT_h}s_T(\underline{u}_T,\underline{u}_T)^{1/2}\lVert \phi\rVert_{H^1(T)}.
\end{eqnarray*}
Hence,
\begin{equation*}
\lVert \fR_h\rVert_{-1}\lesssim(\eta_2^2+\eta_3^2)^{\frac{1}{2}}.
\end{equation*}
\end{proof}
\par Note that if we establish a computable lower bound for $\langle\sigma(u)-F_{\sigma_h\underline{u}_h},u-{u}^*_h\rangle$, then it is clear from Lemma \ref{l1} that the reliability estimate follows. In the proof of the following lemma, we will see that the smoothing function $E^{k+1}_h$, defined in Subsection \ref{ss2.5}, plays an important role.
\begin{lemma}\label{l3}
Let $\epsilon >0$ be an arbitrary number. Then, it holds that
\begin{eqnarray*}
\langle\sigma(u)-F_{\sigma_h\underline{u}_h},u-u^*_h\rangle_{-1} \geq -\frac{\epsilon}{2}\lVert \sigma(u)-F_{\underline{\sigma}_h}\rVert^2_{-1}-\frac{C}{2\epsilon}\lVert\nabla_h(\chi-u^*_h)^+\rVert^2+\sum_{T\in \mathbb F_h}\int_T \sigma_T(\chi-u^*_h)^{-}dx.
\end{eqnarray*}
\end{lemma}
\begin{proof}
Let $\tilde{u}_h=\max\{u^*_h,\chi\}$. Then $\tilde{u}_h\in K$ and using \eqref{ssign} we get,\\ 
\begin{align*}\langle\sigma(u),\tilde{u}_h-u\rangle_{-1}\leq 0.
\end{align*}
\par Now,\\
\begin{align*}
\langle \sigma(u),u-u^*_h\rangle_{-1}&=\langle \sigma(u),u-\tilde{u}_h\rangle_{-1}+\langle \sigma(u),\tilde{u}_h-u^*_h\rangle_{-1}\\
&\geq\langle \sigma(u),\tilde{u}_h-u^*_h\rangle_{-1}=\langle \sigma(u),(\chi-u^*_h)^+\rangle_{-1}\\
&=\langle \sigma(u)-F_{\underline{\sigma}_h},(\chi-u^*_h)^+\rangle_{-1}+\langle F_{\underline{\sigma}_h},(\chi-u^*_h)^+\rangle_{-1}\\
&\geq-\frac{\epsilon}{2}\lVert \sigma(u)-F_{\underline{\sigma}_h}\rVert^2_{-1}-\frac{C}{2\epsilon}\lVert\nabla _h(\chi-u^*_h)^+\rVert^2+\langle F_{\underline{\sigma}_h},(\chi-u^*_h)^+\rangle_{-1},
\end{align*}
\par where $\epsilon>0$ is any arbitrary number. Therefore,\\
\begin{eqnarray}\label{mam1}
\langle\sigma(u)-F_{\underline{\sigma}_h},u-u^*_h\rangle_{-1}\geq-\frac{\epsilon}{2}\lVert \sigma(u)-F_{\underline{\sigma}_h}\rVert^2_{-1}-\frac{C}{2\epsilon}\lVert\nabla_h(\chi-u^*_h)^+\rVert^2\notag\\+ \langle F_{\underline{\sigma}_h},(\chi-u^*_h)^+-(u-u^*_h)\rangle_{-1}.
\end{eqnarray}
\par Now, \\
\begin{align}\label{mam2}
\langle F_{\underline{\sigma}_h},(\chi-u^*_h)^+-(u-u^*_h)\rangle_{-1}&=\sum_{T\in \cT_h}\int_T \sigma_T\left[u^*_h-u+(\chi-u^*_h)^+\right]~dx\notag\\
&=\sum_{T\in \mathbb N_h}\int_T \sigma_T\left[u^*_h-u+(\chi-u^*_h)^+\right]~dx\notag\\&\quad\quad\quad\quad+\sum_{T\in \mathbb F_h}\int_T \sigma_T\left[u^*_h-u+(\chi-u^*_h)^+\right]~dx\notag\\
&\geq \sum_{T\in\mathbb F_h}\int_T  \sigma_T(\chi-u^*_h)^-dx
\end{align}
\par where in obtaining the last inequality we have used that $\sigma_T=0$ for $T\in \mathbb N_h$ from Lemma \ref{l2.2n} and\\
\begin{align*}
\int\limits_T  \sigma_T\left[u^*_h-u\right]~dx=\int\limits_T  \sigma_T\left[u^*_h-\chi+\chi-u\right]~dx\geq \int\limits_T  \sigma_T\left[u^*_h-\chi\right]~dx.
\end{align*}
\par Hence, using equation \eqref{mam1} and \eqref{mam2} we conclude:\\
\begin{align*}
\langle\sigma(u)-F_{\underline{\sigma}_h},u-u^*_h\rangle_{-1}\geq -\frac{\epsilon}{2}\lVert\sigma(u)-F_{\underline{\sigma}_h}\rVert^2_{-1}-\frac{C}{2\epsilon}\lVert\nabla_h(\chi-u^*_h)^+\rVert^2+\sum_{T\in \mathbb F_h}\int_T \sigma_T(\chi-u^*_h)^{-}dx.
\end{align*}
\end{proof}
\par
Using Lemma \ref{l1}, Lemma \ref{l2} and Lemma \ref{l3}, we deduce the following reliability estimate of the error estimator $\eta_h$.\par
\begin{theorem}\label{t1}
Let $u$ and ${\underline{u}_h}$ be the solutions of \eqref{2} and \eqref{14}, respectively. Then, the following holds:
\begin{align*}
&\lVert \nabla_h(u-p_h^{k+1}\underline{u}_h)\rVert^2+\lVert \sigma(u)-F_{\underline{\sigma}_h}\rVert^2_{-1}\lesssim \eta_h^2,
\end{align*}
where, 
\begin{align*}
\eta_h^2=\eta_1^2+\eta_2^2+\eta_3^2+\lVert\nabla_h(\chi-{u}^*_h)^+\rVert^2-\sum\limits_{T\in \mathbb F_h}\int\limits_T \sigma_T(\chi-{u}^*_h)^{-}~dx.
\end{align*}
\end{theorem}  
\subsection{Efficiency of the Error Estimator}~
\par
Let a mesh element $T\in \cT_h$ be fixed and define the following sets of faces and elements sharing at least one node with $T$ :
$$\cF_{\cN,T}\coloneqq\{F\in \cF_h~|~\overline{F}\cap\partial T\neq \emptyset\},~~~\cT_{\cN,T}\coloneqq \{T'\in \cT_h~|~{\overline{T}}'\cap\overline{T}\neq \emptyset\}.$$
For $T\in\cT_h$, define $data~ oscillation$ as $$\text{osc}_h(f;T)\coloneqq\lVert h_T(\bar{f}-f)\rVert_{T},$$
where $\bar{f}\in L^2(\O)$ with $\bar{f}|_{T}=\frac{1}{\lvert T\rvert}\int\limits_T f~dx$.\\
\begin{theorem}\label{t2}
The following estimates hold:\\
For all $T\in \cT_h,$
\begin{equation}
\lVert\nabla(p^{k+1}_T\underline{u}_T-u^*_h)\rVert_T\lesssim \big(\lVert \nabla_h(p^{k+1}_h\underline{u}_h-u)\rVert_{\cT_{\cN,T}}+\lVert\underline{u}_h\rVert_{a,\cT_{\cN,T}}\big),\label{26}
\end{equation}
\begin{multline}
h_T\lVert(f+\Delta p^{k+1}_T\underline{u}_T-\sigma_T)-\pi^0_T(f+\Delta p^{k+1}_T\underline{u}_T-\sigma_T)\rVert_T\lesssim \big(\lVert\nabla(p^{k+1}_T\underline{u}_T-u)\rVert_T\\+ \lVert\sigma(u)-F_{\underline{\sigma}_h}\rVert_{H^{-1}(T)}+osc_h(f,T)\big),\label{27}
\end{multline}
\begin{align}
\lVert\nabla(\chi-u^*_h)^+\rVert_T \lesssim \big(\lVert \nabla_h(p^{k+1}_h\underline{u}_h-u)\rVert_{\cN,T}+\lVert\underline{u}_h\rVert_{a,\cT_{\cN,T}}+\lVert \nabla(\chi-\pi^0_T\chi)\rVert_T\big),\label{29}
\end{align}
\begin{multline}
\int_T -\sigma_T(\chi-u^*_h)^-dx\lesssim\bigg(\lVert\nabla(p^{k+1}_T\underline{u}_T-u)\rVert_T+ \lVert\sigma(u)-F_{\underline{\sigma}_h}\rVert_{H^{-1}(T)}\\+Osc(f,T)+h_T^{-1}\lVert (\chi-u^*_h)^-\rVert_T\bigg)^2+\int_T\big(-f-\Delta p^{k+1}_T\underline{u}_T\big)(\chi-u^*_h)^-dx.\label{30}
\end{multline}
\end{theorem}     
\begin{proof}
Let $T\in\cT_h$ be fixed. Then from Lemma \ref{l2.1}, standard interpolation estimates and equation \eqref{ieq2} we obtain\\
\begin{align*}
\lVert\nabla(p^{k+1}_T\underline{u}_T-{u}^*_h)\rVert_T&\lesssim \sum\limits_{F\in\widetilde{\cF}_T}h^{-1/2}_F\lVert \sjump{p^{k+1}_T\underline{u}_T}\rVert_F\\
&\lesssim\sum\limits_{F\in\widetilde{\cF}_T}\left[h^{-1/2}_F\left\Vert \sjump{p^{k+1}_T\underline{u}_T-u}-\pi^k_F\sjump{p^{k+1}_T\underline{u}_T-u}\right\Vert_F+h^{-1/2}_F\left\Vert \pi^k_F\sjump{p^{k+1}_T\underline{u}_T}\right\Vert_F\right]\\
&\lesssim\sum\limits_{F\in\widetilde{\cF}_T}h^{-1/2}_F\left\Vert \sjump{({p^{k+1}_T\underline{u}_T-u})-\pi^k_F(p^{k+1}_T\underline{u}_T-u)}\right\Vert_F+\\
&\quad\quad\quad\quad\sum\limits_{F\in\tilde{\cF}_T}h^{-1/2}_F\left\Vert \pi^k_F\sjump{p^{k+1}_T\underline{u}_T-u_F}\right\Vert_F\\
&\lesssim\sum\limits_{T\in\cT_{\cN,T}} \left\Vert \nabla(p^{k+1}_h\underline{u}_h-u)\right\Vert_T+\left\Vert \underline{u}_h\right\Vert_{a,\cT_{\mathcal N,T}}.
\end{align*}
\par In order to prove \eqref{27}, the triangle inequality yields, \\
\begin{align*}
h_T\lVert f+\Delta p^{k+1}_T\underline{u}_T-\sigma_T\rVert_T\leq h_T\lVert\bar{f}+\Delta p^{k+1}_T\underline{u}_T-\sigma_T\rVert_T+\text{osc}_h(f;T).
\end{align*}
\par By \cite[Lemma 3.3]{16}, there exists a $\varphi_h\in P^1(T)$ such that $\lVert \varphi_h\rVert_T=1$ and \\
\begin{align*}
\lVert \bar{f}+\Delta p^{k+1}_T\underline{u}_T-\sigma_T\rVert_T\lesssim\int_T (\bar{f}+\Delta p^{k+1}_T\underline{u}_T-\sigma_T)\psi_T\varphi_h~dx,
\end{align*}
\par where $\psi_T\in H^1_0(T)$, is the element bubble function equal to the product of barycentric coordinates of $T$ and rescaled so as to take the value $1$ at the center of gravity of $T$.
\par
Then using \eqref{3}, \eqref{21} and integration by parts we get,\\
\begin{align*}
\int_T (\bar{f}+\Delta p^{k+1}_T\underline{u}_T-\sigma_T&) \psi_T\varphi_h~dx\\ &=\begin{multlined}\int_T (f+\Delta p^{k+1}_T\underline{u}_T-\sigma_T)\psi_T\varphi_h~dx+\int_T (\bar{f}-f)\psi_T\varphi_h~dx\end{multlined}\\
&=(\nabla u,\nabla\psi_T\varphi _h)_T-(\nabla p^{k+1}_T\underline{u}_T,\nabla\psi_T\varphi_h)_T+\int_T(\bar{f}-f)\psi_T\varphi_h~dx+\\& \quad\quad\quad\quad\langle\sigma(u)-F_{\underline{\sigma}_h},\psi_T\varphi_h\rangle_{-1}\\
&\lesssim \big(\lVert \nabla(u-p^{k+1}_T\underline{u}_T)\rVert_T+\lVert \sigma(u)-F_{\underline{\sigma}_h}\rVert_{H^{-1}(T)}\big)\lVert \nabla(\psi_T\varphi_h)\rVert_T+\\&\quad\quad\quad\quad\lVert \bar{f}-f\rVert_T\lVert\psi_T\varphi_h\rVert_T.
\end{align*}
\par Since,\\
\begin{align*}
\lVert\nabla(\psi_T\varphi_h)\rVert_T\lesssim h^{-1}_T\lVert\psi_T\varphi_h\rVert_T\lesssim h^{-1}_T.
\end{align*}
\par Therefore,\\
\begin{align*}
h_T\lVert f+\Delta p^{k+1}_T\underline{u}_T-\sigma_T\rVert_T\lesssim \lVert\nabla(p^{k+1}_T\underline{u}_T-u)\rVert_T+\lVert\sigma(u)-F_{\underline{\sigma}_h}\rVert_{H^{-1}(T)}+\text{osc}_h(f,T).
\end{align*}
\par Hence,\\
\begin{multline*}
h_T\lVert(f+\Delta p^{k+1}_T\underline{u}_T-\sigma_T)-\pi^0_T(f+\Delta p^{k+1}_T\underline{u}_T-\sigma_T)\rVert_T\lesssim \bigg(\lVert\nabla(p^{k+1}_T\underline{u}_T-u)\rVert_T+\\ \lVert\sigma(u)-F_{\underline{\sigma}_h}\rVert_{H^{-1}(T)}+\text{osc}_h(f,T)\bigg),
\end{multline*}
\par where we have used the fact that $\lVert g-\pi^0_Tg\rVert_T\leq \lVert g\rVert_T~\forall ~g\in L^2(T)$.\\ \\
In order to prove \eqref{29}, using the fact that $\chi,~u^*_h$ are continuous inside $\O$ and $0\leq \lvert \nabla(\chi-u^*_h)\rvert$, $\pi^0_T(\chi)\in \mathbb P^0(T)~\forall T\in\cT_h$, we get:\\
\begin{eqnarray*}
\lVert\nabla(\chi-u^*_h)^+\rVert_T &=&\left(\int_{T\cap\{\chi>u^*_h\}}\lvert \nabla(\chi-u^*_h)\rvert^2~dx+\int_{T\setminus\{\chi>u^*_h\}}0\right)^{\frac{1}{2}}\\
&\leq&\lVert \nabla(\chi-u^*_h)\rVert_T\\
&\leq&\lVert \nabla(\chi-\pi^0_T\chi)\rVert_T+\lVert \nabla(p^{k+1}_T\underline{u}_T-u^*_h)\rVert_T+\lVert \nabla(\pi^0_T\chi-p^{k+1}_T\underline{u}_T)\rVert_T\\
&=&\lVert \nabla(\chi-\pi^0_T\chi)\rVert_T+\lVert \nabla(p^{k+1}_T\underline{u}_T-u^*_h)\rVert_T+\lVert \nabla p^{k+1}_T\underline{u}_T\rVert_T\\
&\lesssim&\lVert \nabla(\chi-\pi^0_T\chi)\rVert_T+\lVert \nabla(p^{k+1}_T\underline{u}_T-u^*_h)\rVert_T+\lVert  \underline{u}_{T}\rVert_{a,T}.
\end{eqnarray*}
\par In order to prove \eqref{30},
\begin{align*}
-\int_T\sigma_T(\chi-u^*_h)^-dx&=\int_T\big(f+\Delta p^{k+1}_T\underline{u}_T-\sigma_T\big)(\chi-u^*_h)^-dx+\\&\quad\quad\quad\quad \int_T\big(-f-\Delta p^{k+1}_T\underline{u}_T\big)(\chi-u^*_h)^-dx\\
&\lesssim h_T^2\lVert f+\Delta p^{k+1}_T\underline{u}_T-\sigma_T\rVert_T^2+h_T^{-2}\lVert (\chi-u^*_h)^-\rVert_T^2+\\&\quad\quad\quad\quad  \int_T\big(-f-\Delta p^{k+1}_T\underline{u}_T\big)(\chi-u^*_h)^-dx.
\end{align*}
\par Hence \eqref{30} is now proved using the above inequality and estimates of \eqref{27}.
\end{proof}
\section{Numerical Experiments}\label{sec4}
In this section, we present a series of numerical results to demonstrate the effectiveness of the $a~ posteriori$ error estimator developed in Section \ref{sec3}. Although the analysis has been conducted based on the model problem \eqref{2} with homogeneous boundary conditions, it is worth noting that the that the analysis can easily be extended to the underlying problem with non-homogeneous boundary conditions. We employ a standard adaptive algorithm, consisting of the following sequential steps:
\begin{enumerate}
\item {\bf SOLVE:} The solution to the discrete nonlinear problem (equation \eqref{14}) is obtained using the primal-dual active set algorithm \cite{29}.
\item {\bf ESTIMATE:} In this step, we compute the $a~ posteriori$ error estimator $\eta_h$ for each element $T \in \cT_h$.
\item {\bf MARK:} We utilize the error estimations obtained in the ESTIMATE step and apply the D\"orfler marking strategy \cite{30} with a specified parameter $\Theta=0.3$ for marking elements.
\item {\bf REFINE:} Finally, we perform mesh refinement using the newest vertex bisection algorithm, resulting in the generation of a new adaptive mesh.
\end{enumerate}
This process allows us to validate the performance of our error estimator $\eta_h$, defined in Theorem \ref{t1}, for various obstacle problems. In the numerical experiments, we have taken the discrete solution to be constant polynomial in cell unknowns and linear polynomial on the faces. This makes our potential reconstruction operator to be an elementwise polynomial of degree 2.

\begin{example}\label{exm1}
We consider the square domain $\O=(-1,1)^2$ and the obstacle function $\chi=0$. We prescribe a contact radius $r_0=0.7$ and, setting $r^2= x^2 + y^2$, we take the load function
$$f(x,y)\coloneqq \begin{cases} -4(4r^2-2r_0^2),& \text{if}~ r>r_0,\\
						-8r_0^2(1-r^2+r_0^2),& \text{if}~ r\leq r_0.\\
						\end{cases}$$
Figure \ref{f4.11} illustrate the behavior of the energy error $\lVert \nabla(u-p^{k+1}_h\underline{u}_h)\rVert$ and $a~ posteriori$ error estmator $\eta_h$ with respect to Degrees of Freedoms (DOFs). This figure ensures the optimal convergence (rate DOFs$^{-1}$ ) of the error and the estimator together with the reliability of the estimator. Figure \ref{f4.12} represent efficiency indices with respect to degrees of freedom. Fig \ref{f4.13} displays the mesh refinement at level 31. 
\end{example}

\begin{figure}
	\begin{subfigure}[b]{0.4\textwidth}
		\includegraphics[width=\linewidth]{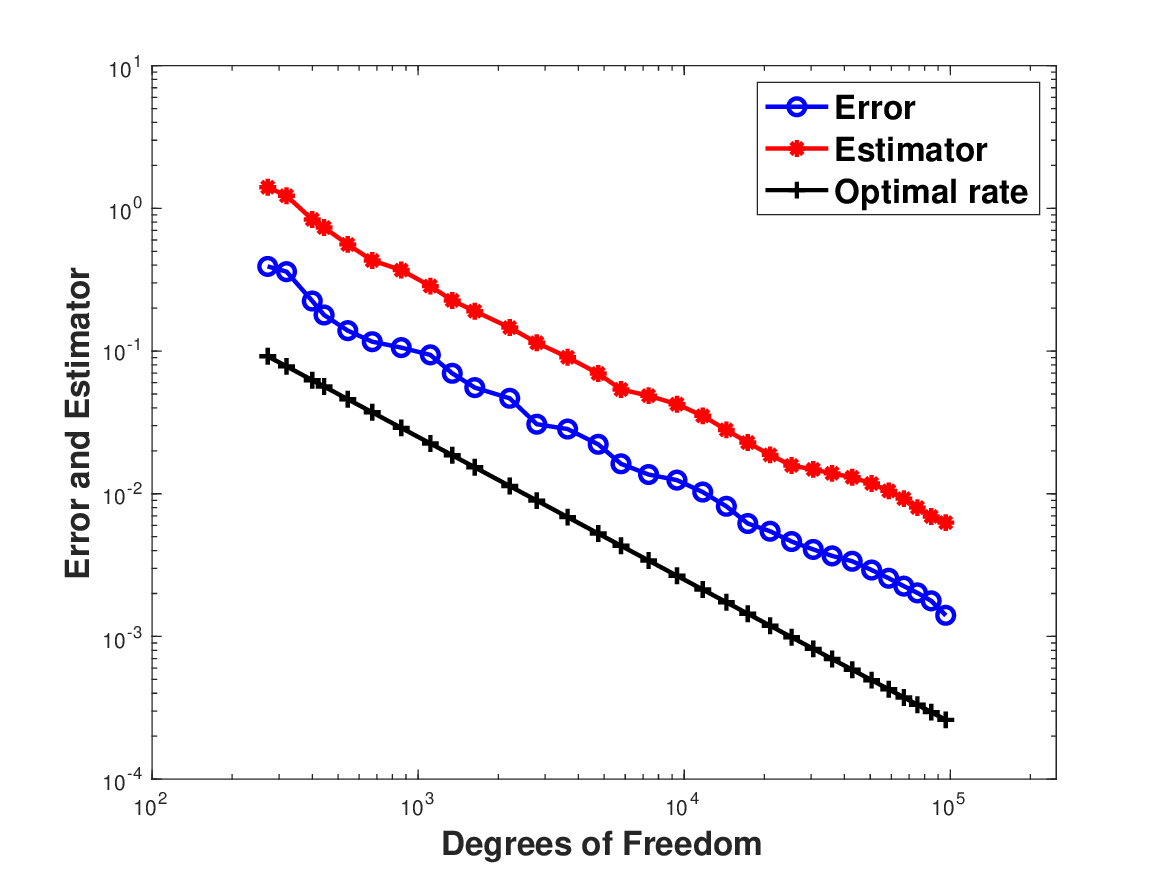}
		\caption{}
		\label{f4.11}
	\end{subfigure}
	\begin{subfigure}[b]{0.4\textwidth}
		\includegraphics[width=\linewidth]{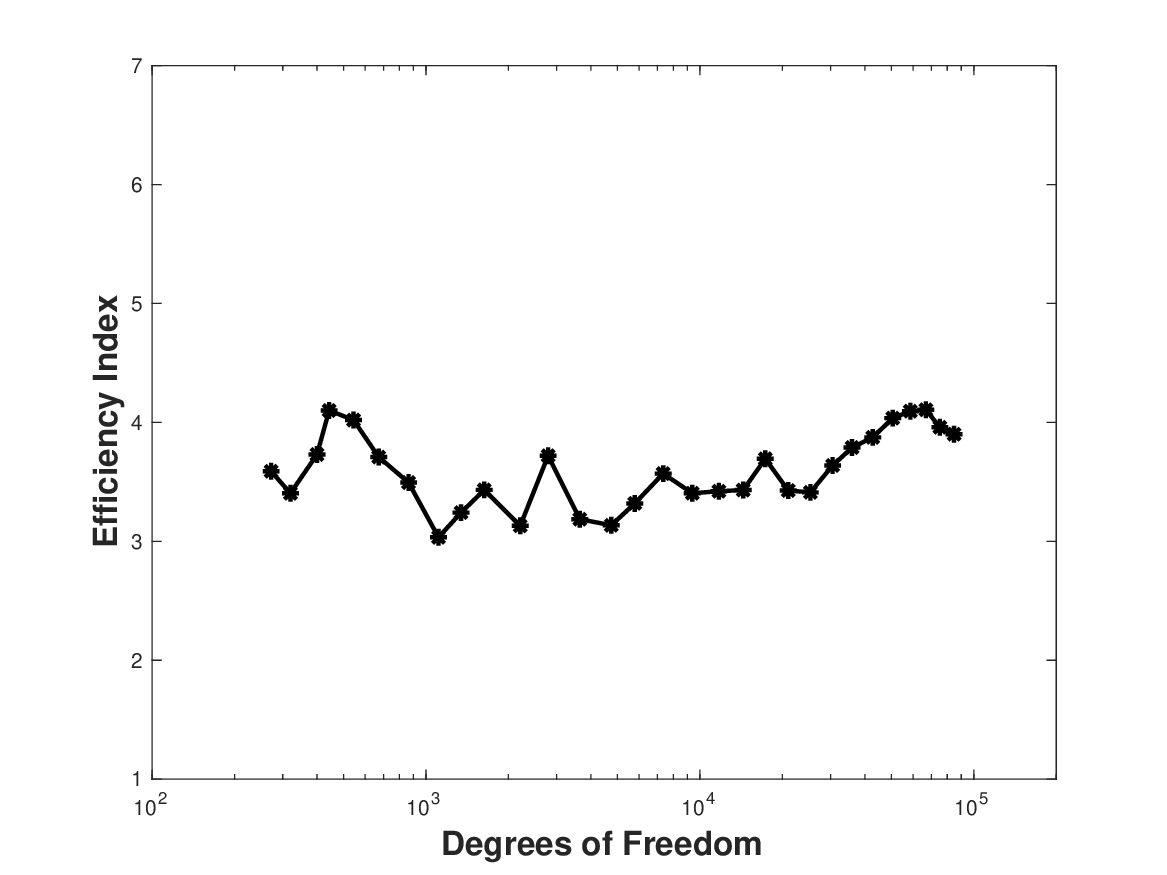}
		\caption{}
		\label{f4.12}
	\end{subfigure}
	\caption{Error, Estimator and Efficiency Index for Example 1.}\label{f4.1}
\end{figure} 

\begin{figure}
	\begin{subfigure}[b]{0.4\textwidth}
		\includegraphics[width=\linewidth]{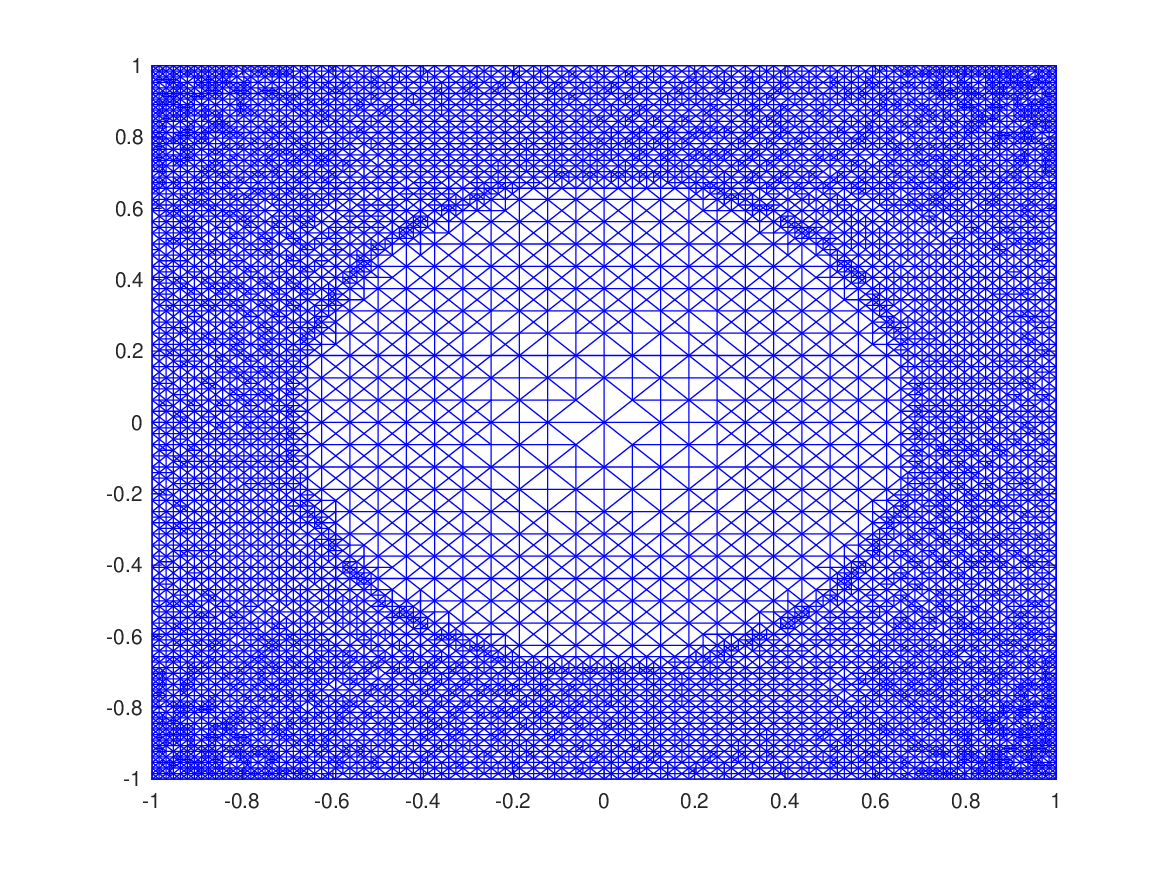}
		\caption{}
		\label{f4.13}
	\end{subfigure}
	\begin{subfigure}[b]{0.4\textwidth}
		\includegraphics[width=\linewidth]{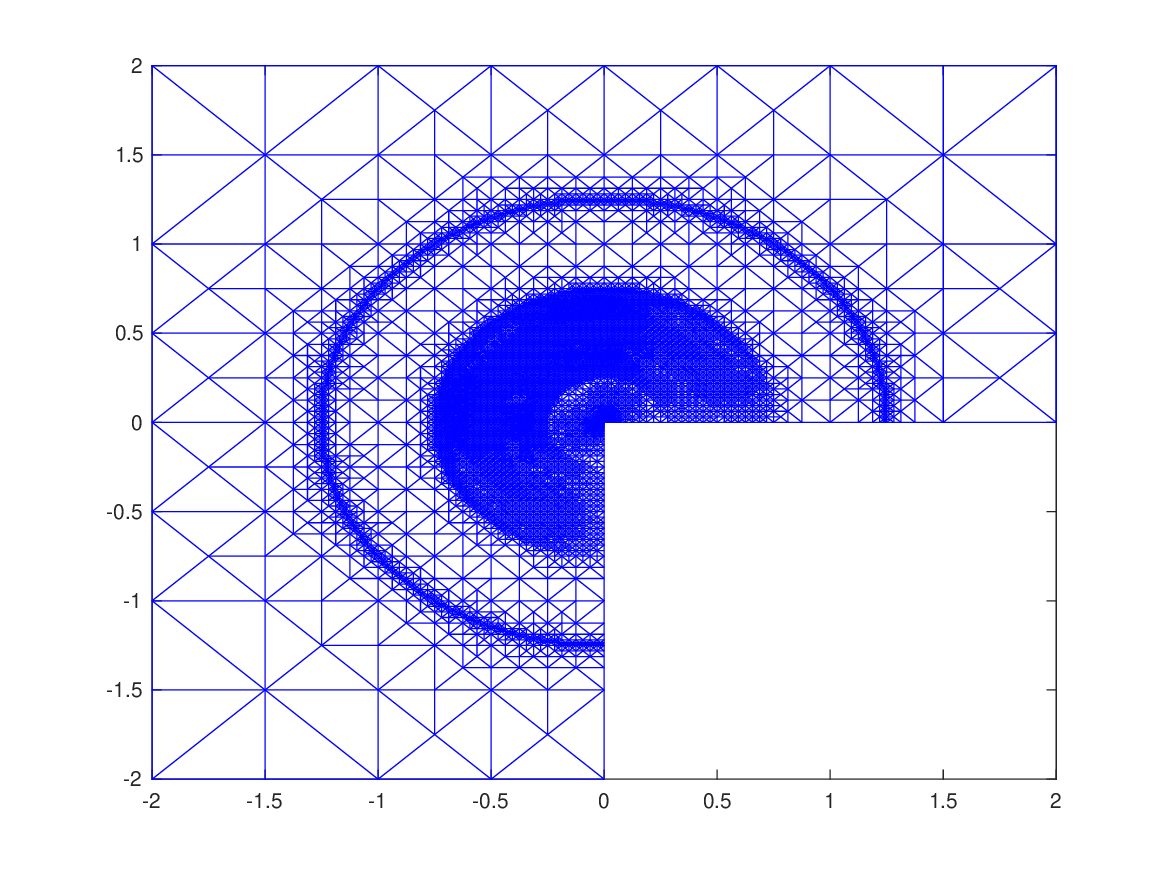}
		\caption{}
		\label{f4.23}
	\end{subfigure}
	\caption{Adaptive Mesh for Example 1 and Example 2 respectively.}\label{f4.2}
\end{figure}
\begin{figure}
	\begin{subfigure}[b]{0.4\textwidth}
		\includegraphics[width=\linewidth]{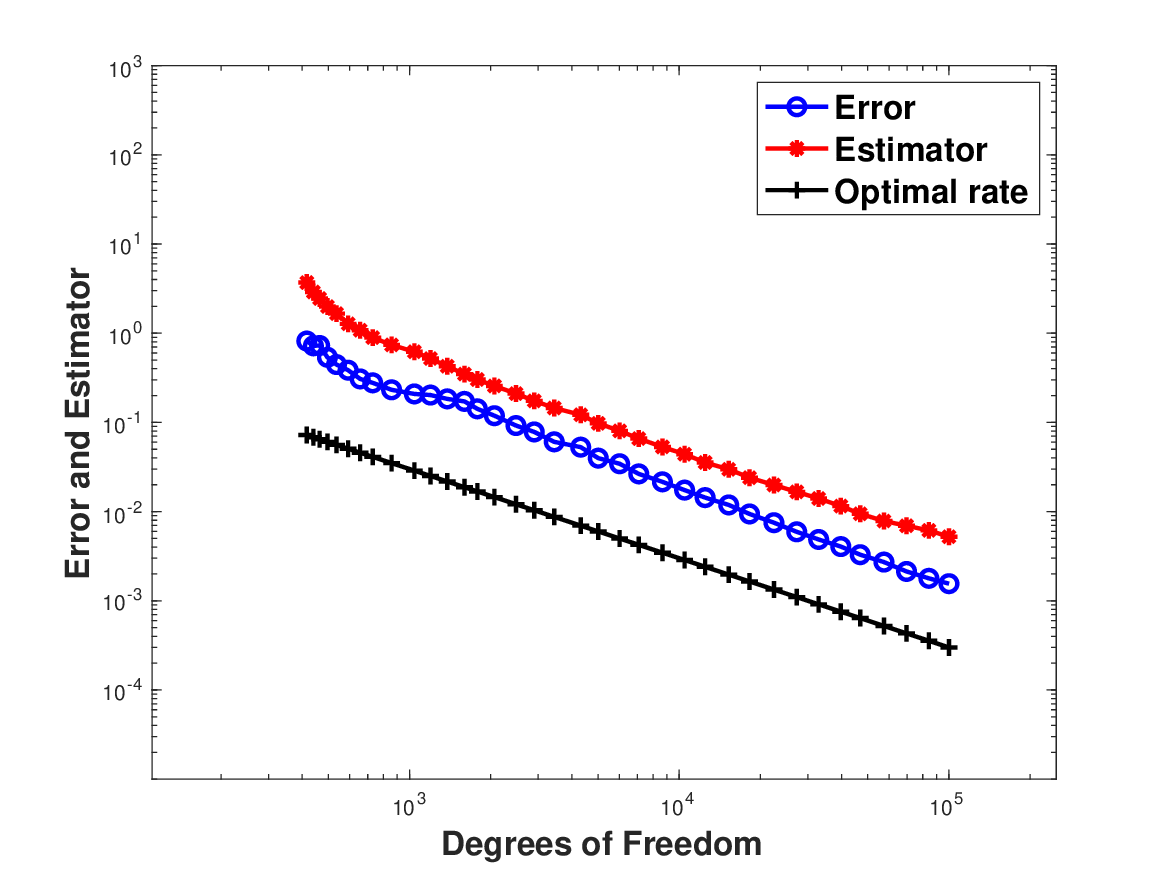}
		\caption{}
		\label{f4.21}
	\end{subfigure}
	\begin{subfigure}[b]{0.4\textwidth}
		\includegraphics[width=\linewidth]{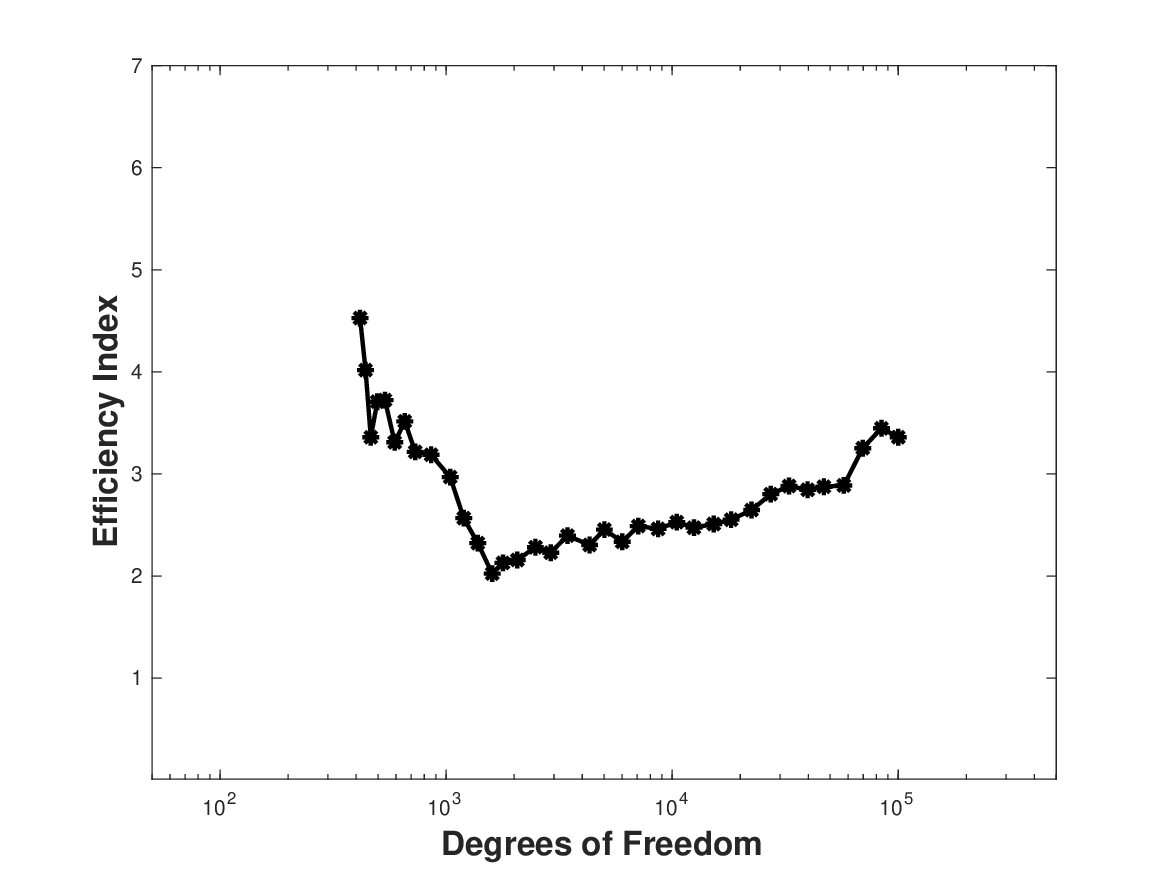}
		\caption{}
		\label{f4.22}
	\end{subfigure}
	\caption{Error, Estimator and Efficiency Index for Example 2.}\label{f4.3}
\end{figure}

\begin{example}\label{exm2}
In this example, we consider the domain to be non-convex. Let $\Omega = (-2,2)^2 \setminus [ 0, 2) \times (-2, 0]$ and $ \chi=0$. Therein, the given data is as follows
	\begin{align*}
	u &= r^{2/3} sin(2\theta/3 ) \gamma_1(r),\\
	f&=-r^{2/3} sin(2 \theta/3) \Big(\frac{\gamma_1^{'}(r)}{r} + \gamma_1^{''}(r) \Big)-\frac{4}{3}r^{-1/3}sin(2 \theta /3) \gamma_1^{'}(r)-\gamma_2(r),\\
	\text{where}\\
	\gamma_1(r) &= \begin{cases}  1, \quad  \tilde{r}<0 \\
	-6 \tilde{r}^5 + 15 \tilde{r}^4 -10 \tilde{r}^3 +1 , \quad  0 \leq \tilde{r} <1 \\
	0  , \quad \tilde{r} \geq 1,\\
	\end{cases}\\
	\gamma_2(r) &= \begin{cases} 0, \quad r \leq \frac{5}{4} \\
	1, \quad \mbox{otherwise}
	\end{cases}\\
	\text{with}\quad
	\tilde{r}= 2(r-1/4).
	\end{align*}
	\noindent
	The behavior of the true error and error estimator is reported in Figure \ref{f4.21}. We observe that both error and estimator converge with the optimal rate DOFs$^{-1}$. The efficiency of the error estimator is shown in Figure \ref{f4.22}. Figure \ref{f4.23} displays the mesh refinement at level 36 indicating that the singular behavior of the solution have been captured very well by the error estimator.
\end{example}



\newpage
\begin{thebibliography}{100}
\addcontentsline{toc}{section}{References}

\bibitem{8}
M. Ainsworth and J. T. Oden; $A~ Posteriori$ error estimation in finite element analysis, Wiley-Blackwell, 2000.

\bibitem{17}
K. Atkinson and W. Han; Theoretical Numerical Analysis. A functional analysis framework, Third edition, Springer, 2009.

\bibitem{34}
S. Bartels and C. Carstensen; Averaging techniques yield reliable $a~ posteriori$ finite element error control for obstacle problems, {Numerische Mathematik}, 2004, 99, 225-249.

\bibitem{12}
F. B. Belgacem; Numerical solution of some variational inequalities arisen from unilateral contact problems by finite element method, {SIAM Journal on Numerical Analysis}, 2000, 37, 1198-1216.

\bibitem{35}
D. Braess; $A~ posteriori$ error estimators for obstacle problems-another look, {Numerische Mathematik}, 2005, 101, 415-421.

\bibitem{31}
S. C. Brenner; Convergence of nonconforming multigrid methods without full elliptic regularity, Mathematics of computation, 1999, 68, 25-53.

\bibitem{32}
S. C. Brenner, L. Y. Sung and Y. Zhang; Finite element methods for the displacement obstacle problem of clamped plates, Mathematics of Computation, 2012, 81, 1247-1262.

\bibitem{bren}
S. C. Brenner and L. R. Scott; The Mathematical Theory of Finite Element Methods, third edition, Texts in Applied Mathematics. Springer, New York, 2008, 15.

\bibitem{33}
F. Brezzi, W. W. Hager and P. A. Raviart; Error estimates for the finite element solution of variational inequalities, Part I. primal theory, {Numerische Mathematik}, 1977, 28, 431-443.

\bibitem{24}
F. Brezzi, K. Lipnikov and M. Shashkov; Convergence of the mimetic finite difference method for diffusion problems on polyhedral meshes, SIAM Journal on Numerical Analysis, 2005, 43(5), 1872-1896.

\bibitem{25}
F. Brezzi, K. Lipnikov, M. Shashkov and V. Simoncini; A new discretization methodology for diffusion problems on generalized polyhedral meshes, Computer Methods in Applied Mechanics Engineering, 2007, 196, 3682-3692.

\bibitem{20}
E. Burman and A. Ern; An unfitted Hybrid High-Order method for elliptic interface problems, SIAM Journal on Numerical Analysis, 2018, 56(3), 1525-1546.

\bibitem{4}
F. Chouly, A. Ern and N. Pignet; A hybrid high-order discretization combined with Nitsche's method for contact and tresca friction in small strain elasticity, SIAM Journal on Scientific Computing, 2020, 42, 2300-2324.

\bibitem{5}
P. G. Ciarlet; The Finite Element Method for Elliptic Problems, SIAM, 2002.

\bibitem{3}
M. Cicuttin, A. Ern and T. Gudi; Hybrid high-order methods for the elliptic obstacle problem, Journal of Scientific Computing, 2020, 83(8).

\bibitem{27}
B. Cockburn, J. Gopalakrishnan and R. Lazarov; Unified hybridization of discontinuous Galerkin, mixed, and continuous Galerkin methods for second order elliptic problems, SIAM Journal on Numerical Analysis, 2009, 47(2), 1319-1365.

\bibitem{26}
B. Cockburn, D. A. Di Pietro and A. Ern; Bridging the Hybrid High-Order and Hybridizable discontinuous Galerkin methods, ESAIM: Mathematical Modelling and Numerical Analysis, 2016, 50(3), 635-650.

\bibitem{28}
B. A. D. Dios, K. Lipnikov and G. Manzini; The nonconforming virtual element method, ESAIM: Mathematical Modelling and Numerical Analysis, 2016, 50(3), 879-904.

\bibitem{30}
W. D\"orlfer; A convergent adaptive algorithm for Poisson's equation, SIAM Journal on Numerical Analysis, 1996, 33, 1106-1124.

\bibitem{23}
R. Eymard, T. Gallouet, and R. Herbin; Discretization of heterogeneous and anisotropic diffusion problems on general nonconforming meshes SUSHI: a scheme using stabilization and hybrid interfaces, IMA Journal of Numerical Analysis, 2010, 30(4), 1009-1043.

\bibitem{14}
R. S. Falk;  Error Estimation for the approximation of a class of Variational Inequalities, Mathematics of Computation, 1974, 28, 963-971.

\bibitem{36}
S. Gaddam and T. Gudi; Inhomogeneous Dirichlet boundary condition in the $a~ posteriori$ error control of the obstacle problem, Computers and Mathematics with Applications, 2018, 75(7), 2311-2327.

\bibitem{rs3}
D. Garg, K. Porwal and R. Singla; Adaptive nonconforming finite element method for the Signorini problem in the supremum norm, Communicated.

\bibitem{9}
R. Glowinski; Lectures on Numerical Methods For Non-Linear Variational Problems, Springer, 2008.

 \bibitem{38}
 T. Gudi and K. Porwal; $A~ posteriori$ error control of discontinuous Galerkin methods for elliptic obstacle problems, Mathematics of Computation, 2014, 83(286), 579-602.
 
 \bibitem{39}
T. Gudi and K. Porwal; A remark on the $a~ posteriori$ error analysis of discontinuous Galerkin methods for obstacle problem, Computational Methods in Applied Mathematics, 2014, 14, 71-87. 

\bibitem{37}
T. Gudi and K. Porwal; A reliable residual based $a~ posteriori$ error estimator for quadratic finite element method for the elliptic obstacle problem, Computational Methods in Applied Mathematics, 2015, 15, 145-160.

\bibitem{29}
M. Hintermuller, K. Ito and K. Kunisch; The primal-dual active set strategy as a semi-smooth newton method, SIAM Journal on Optimization, 2002, 13(3), 865-888.

\bibitem{38}
R. Khandelwal, K. Porwal and R. Singla; Supremum-norm $a~ posteriori$ error control of quadratic discontinuous Galerkin methods for the obstacle problem, Computers and Mathematics with Applications, 2023, 137, 147-171.

\bibitem{15}
K. Y. Kim; $A~ posteriori$ error estimators for locally conservative methods of nonlinear elliptic problems, Applied Numerical Mathematics, 2007, 57, 1065-1080.

\bibitem{kind}
D. Kinderlehrer and G. Stampacchia; An introduction to variational inequalities and their applications, SIAM, Philadelphia, 2000.

\bibitem{21}
D. A. Di Pietro and J. Droniou; A Hybrid High-Order method for Leray Lions elliptic equations on general meshes. Mathematics of Computation, 2017, 86(307), 2159-2191.

\bibitem{18}
D. A. Di Pietro, J. Droniou and A. Ern; A discontinuous-skeletal method for advection-diffusion-reaction on general meshes, SIAM Journal on Numerical Analysis, 2015, 53(5), 2135-2157.

\bibitem{6}
D. A. Di Pietro and A. Ern, Mathematical Aspects of Discontinuous Galerkin Methods, Springer, 2012.

\bibitem{2}
D. A. Di Pietro and A. Ern; A hybrid high-order locking-free method for linear elasticity on general meshes, Computer Methods in Applied Mechanics and Engineering, 2015, 283, 1-21.

\bibitem{7}
D. A. Di Pietro, A. Ern and L. Formaggia; Numerical Methods for PDEs State of the Art Techniques, Springer, 2018.

\bibitem{1}
D. A. Di Pietro, A. Ern and S. Lemaire; An Arbitrary-Order and Compact-Stencil Discretization of Diffusion on General Meshes Based on Local Reconstruction Operators, Computational Methods in Applied Mathematics, 2014, 14(4), 461-472.

\bibitem{19}
D. A. Di Pietro, A. Ern, A. Linke and F. Schieweck; A discontinuous skeletal method for the viscosity-dependent Stokes problem, Computer Methods in Applied Mechanics and Engineering, 2016, 306, 175-195.

\bibitem{22}
D. A. Di Pietro and S. Krell; A Hybrid High-Order method for the steady incompressible Navier Stokes problem, Journal of Scientific Computing, 2018, 74(3), 1677-1705.


\bibitem{39}
K. Porwal and R. Singla; Pointwise adaptive non-conforming finite element method for the obstacle problem, Computational and Applied Mathematics, 2024, 43(150).

\bibitem{stamp}
G. Stampacchia; On some regular multiple integral problems in the calculus of variations, Journal of Mathematics and Mechanics, 1967, 16 (4), 875-908.

\bibitem{13}
A. Veeser; Efficient and Reliable $a~ posteriori$ error estimates for elliptic obstacle problems, SIAM Journal on Numerical Analysis, 2001, 39, 146-167.

\bibitem{16}
R. Verfurth; A review of $A~ Posteriori$ Error Estimation and Adaptive Mesh-Refinement Techniques, Wiley-Teubner, Stuttgart, Germany, 1996.

\bibitem{11}
A. Weiss, and B. I. Wohlmuth; $A~ posteriori$ error estimator and error control for contact problems, Mathematics of Computation, 2009, 78, 1237-1267.

\end{thebibliography}
\end{document}